\documentclass[11pt,reqno]{amsart}
\usepackage{amscd,amssymb,amsmath,amsthm}
\usepackage[arrow,matrix]{xy}
\usepackage{graphicx}
\usepackage{cite}
\usepackage{color}
\usepackage{boldline}
\usepackage{multirow}
\newtheorem{thm}[subsection]{Theorem}

\newtheorem{pro}[subsection]{Proposition}

\numberwithin{equation}{section} 

\newcommand{\s}{{\sigma}}

\newcommand{\bea}{\begin{eqnarray}}
\newcommand{\eea}{\end{eqnarray}}

\newcommand{\Z}{\mathbb{Z}}

\newcommand{\R}{\mathbb{R}}

\begin{document}
\title[Phase transitions for a model on the Cayley tree]{Phase transitions for a model with uncountable spin space on the Cayley tree: the general case}

\author{Golibjon Botirov}
\author{Benedikt Jahnel}
\address[Benedikt Jahnel]{Weierstrass Institute Berlin, Mohrenstr. 39, 10117 Berlin, Germany, \texttt{https://www.wias-berlin.de/people/jahnel/}}
\email{Benedikt.Jahnel@wias-berlin.de}

\address[Golibjon Botirov]{National University of Uzbekistan, University Street 4, 100174 Tashkent, Uzbekistan,} \email{botirovg@yandex.ru}

\begin{abstract}
In this paper we complete the analysis of a statistical mechanics model on Cayley trees of any degree, started in~\cite{ehr1, re, erb, bb, b1}. The potential is of nearest-neighbor type and the local state space is compact but uncountable. Based on the system parameters we prove existence of a critical value $\theta_{\rm c}$ such that for $\theta\le \theta_{\rm c}$ there is a unique translation-invariant splitting Gibbs measure. For $\theta_{\rm c}<\theta$ there is a phase transition with exactly three translation-invariant splitting Gibbs measures. The proof rests on an analysis of fixed points of an associated non-linear Hammerstein integral operator for the boundary laws.
\end{abstract}
\maketitle



\section{Introduction} \label{sec:intro}
In the present note we complete a line of research about the phase-transition behaviour of  a nearest-neighbor model on Cayley trees with arbitrary degree $k\ge 2$. As first described in~\cite{re}, for a given consistent family of finite-volume Gibbs measures, the existence and multiplicity of a certain class of infinite-volume measures which are consistent with the prescribed finite-volume Gibbs measures, can be reduced to the analysis of fixed points of some non-linear integral equation of Hammerstein type. Every positive solution of the fixed point equation here corresponds to a measures which is called a splitting Gibbs measure. Every splitting Gibbs measure is also a Gibbs measure in the sense of the DLR formalism; see~\cite{B}. This approach has been successfully applied in the analysis of a variety of different models on Cayley trees with respect to their phase-transition properties; see~\cite{rb} for a comprehensive overview. In particular, starting with~\cite{ehr1}, a phase-transition of multiple splitting Gibbs measures has been detected in a model with uncountable local state space $[0,1]$ and nearest-neighbor interactions. This has motivated the subsequent analysis in~\cite{erb, bb, b1}, to further understand critical behavior of this model for all degrees of the underlying tree, where also new parameters are introduced. It is the purpose of this note to complete the analysis of this model.

\medskip
For nearest neighbors $x,y$ on the {\it Cayley tree} $\Gamma^k$ with degree $k \geq 2$ with local states $\s(x),\s(y)\in[0,1]$, we consider the potential
\begin{equation}\label{Model}
\xi_{\s(x),\s(y)}=\log\left(1+\theta \sqrt[2m+1]{4(\s(x)-\frac{1}{2})(\s(y)-\frac{1}{2})}\right)
\end{equation}
where $m\in\mathbb N\cup\{0\}$, $0 \leq \theta <1$ are the system parameters. It can be interpreted as a certain symmetric pair-interaction with values in $[\log(1-\theta),\log(1+\theta)]$, admitting two distinct ground states given by the all-$0$ and the all-$1$ configuration.  The main result is the existence of a sharp threshold
$$\theta_{\rm c}=\frac{2m+3}{k(2m+1)}$$
such that if $\theta_{\rm c}< \theta <1$, there are exactly three translation-invariant splitting Gibbs measures and otherwise there is only one.

%
%
%
%
%

\section{Setup}\label{Setup}
\subsection{Gibbs measures on Cayley trees}

The {\it Cayley tree} $\Gamma^k$ of order $k \geq 1$ is an infinite tree, i.e., a graph
without cycles, such that exactly $k+1$ edges originate from each vertex. Let $\Gamma^k=(V,L)$ where
$V$ is the set of {\it vertices} and $L$ is a symmetric subset of $V \times V$, called the {\it edge set}.
The word "symmetric" means that $(x, y) \in L$ iff $(y, x) \in L$. Here, $x$ and $y$ are called the {\it endpoints} of the edge $\langle x, y\rangle$. Two vertices $x$ and $y$ are called {\it nearest neighbors}
if there exists an edge $l \in L$ connecting them and we denote $l=\langle x,y\rangle$. 
For a fixed $x^0 \in V$, called the {\it root}, we defines $n$-spheres and $n$-disks in the graph distance $d(x,y)$ by  
$$W_n=\{x \in V| d(x,x^0)=n\}, \ \ \ \ V_n=\bigcup \limits_{i=0}^n W_i$$
and denote for any $x \in W_n$ the set of {\it direct successors} of $x$ by
$$S(x)=\{y \in W_{n+1}: d(x,y)=1\}.$$

For $A \subset V$ let $\Omega_A=[0,1]^A$ denote the set of all configurations $\s_A$ on $A$. In particular, a configuration $\sigma$ on $V$ is then defined as a function $V\ni x \mapsto \sigma(x) \in [0,1]$.
According to the usual setup for Gibbs measure, we consider a (formal) Hamiltonian of the form
\begin{equation}\label{m}
H(\sigma)=-\sum \limits_{\langle x,y\rangle \in L}\xi_{\sigma(x),\sigma(y)}, 
\end{equation}
where $\xi: (u,v) \in [0,1]^2 \mapsto \xi_{u,v} \in\R$ is the interaction~\eqref{Model} which assigns energy only to neighboring sites. Since $\xi$ does not depend on the locations $x$ and $y$, $H$ is invariant under tree translations. 
Let $\lambda$ be the Lebesgue measure on $[0,1]$ then, on the set of all configurations on $A$ the a priori measure $\lambda_A$ is introduced as the $|A|$-fold product of the measure $\lambda$. Here and in the sequel, $|A|$ denotes the cardinality of $A$. We equip $\Omega=\Omega_V$ with the standard sigma-algebra $\mathcal{B}$ generated by the cylindrical subsets. A probability measure $\mu$ on $(\Omega, \mathcal{B})$ is called a {\it Gibbs measure} (with Hamiltonian $H$) if it satisfies the DLR equation. That is, for any $n=1,2,...$ and bounded measurable test function $f$, we have that
\begin{equation}\label{DLR}
\int\mu(d\s)f(\s)=\int \mu(d \s)\int\gamma_{V_n}(d\tilde\s_{V_n}|\s_{W_{n+1}})f(\tilde\s_{V_n}\s_{\Gamma^k\setminus V_n}),
\end{equation}
where $\gamma_{V_n}(d \sigma_{V_n}|\s_{\Gamma^k\setminus V_n})$ is the Gibbsian specification
$$\gamma_{V_n}(d\tilde\sigma_{V_n}|\s_{\Gamma^k\setminus V_n})=\frac{1}{Z_{V_n}(\s_{W_{n+1}})}e^{-\beta H(\tilde\sigma_{V_n}\s_{W_{n+1}})}\lambda_{V_n}(d\tilde\s_{V_n}),$$
with normalization $Z_{V_n}$ and temperature parameter $\beta \geq 0$. Such a specification is also sometimes referred to as a Markov specification; see~\cite{Ge11}.

\subsection{Representation via Hammerstein operators}
A subset of the infinite-volume Gibbs measures defined via the DLR equation~\eqref{DLR}, called the {\it splitting Gibbs measures} or {\it Markov chains}, can be represented in terms of the fixed points of some nonlinear integral operator of Hammerstein type; see~\cite{re} for details. More precisely, for every $k\in\mathbb{N}$ consider the integral operator
$H_{k}$ acting on the cone $C^{+}[0,1]=\{f\in C[0,1]: f(x)\geq 0\}$ given by
\begin{equation}\label{Ham_k}
\begin{split}
(H_{k}f)(t)=\int^{1}_{0}K(t,u)f^{k}(u)du.
\end{split}
\end{equation}
Then, the translation-invariant splitting Gibbs measures for the Hamiltonian~\eqref{m} correspond to fixed points of $H_k$ with $K(t,u)=\exp(\beta\xi_{t,u})$, often called {\it boundary laws}. Note that $H_k$ in general might generate ill-posed problems; see~\cite{Kr64,KrZa84}.

\section{Main results}
The main result of this note is the following characterization of phase-transition regimes of the model~\eqref{Model} with $\beta=1$.
\begin{thm}\label{t}
For all $n\in\mathbb N\cup\{0\}$ and $k\ge 2$ let $\theta_{\rm c}=(2n+3)/(k(2n+1))$, then the model~\eqref{Model} has
\begin{enumerate}
\item a unique translation-invariant splitting Gibbs measure if $0 \leq \theta \leq \theta_{\rm c}$ and
\item exactly three translation-invariant splitting Gibbs measures if $\theta_{\rm c} < \theta <1$.
\end{enumerate}
\end{thm}
The proof is based on a characterization of solutions to the fixed point equation for the associated Hammerstein integral operator~\eqref{Ham_k} as given in Proposition~\ref{Prop1} below. In case of the model at hand, then the analysis can be reduced to finding the fixed points of the following $2$-dimensional operator $V_k:(x,y) \in \R^2 \rightarrow
(x', y') \in \R^2$
\begin{equation}\label{V_k}
\begin{split}
V_{k,n}(x,y)=\left\{
\begin{array}{lllllll}
x'=\sum \limits_{j=0}^{\lfloor\frac{k}{2}\rfloor} \binom{k}{2j}
\frac{2n+1}{2n+1+2j} \cdot 2^{\frac{2j}{2n+1}} \cdot x^{k-2j} (\theta y)^{2j}$$
\\ [6 mm]$$y'=\sum \limits_{j=0}^{\lfloor\frac{k}{2}\rfloor} \binom{k}{2j+1}
\frac{2n+1}{2n+2+2j+1} \cdot 2^{\frac{2j}{2n+1}} \cdot x^{k-(2j+1)} (\theta y)^{2j+1}
\end{array}\right.
\end{split}
\end{equation}
with $k \geq 2$, which is then the content of Proposition~\ref{Prop2}.

\begin{pro}\label{Prop1} A function $\varphi \in C[0,1]$ is a solution of the Hammerstein equation
\begin{equation}\label{H_k}
\begin{split}
H_kf=f\end{split}
\end{equation}
with $H_k$ defined in~\eqref{Ham_k} for our model~\eqref{Model}, iff $\varphi$ has the following form
$$\varphi(t)=C_1 + C_2
\theta \sqrt[2n+1]{4(t-\frac{1}{2})},$$
where $(C_1, C_2) \in \R^2$ is
a fixed point of the operator $V_{k,n}$ as defined in~\eqref{V_k}.
\end{pro}
%
%
%
%
In the following proposition we characterize the fixed points of $V_{k,n}$ which readily implies Theorem~\ref{t} using Proposition~\ref{Prop1}. 
\begin{pro}\label{Prop2}
Let $\theta_{\rm c}=(2n+3)/(k(2n+1))$, then there exist $x_o,y_o\in(0,\infty)$ such that the number and form of the fixed points of the operator $V_{k,n}$ are as presented in the following Table~\ref{Table1}.
\begin{table}[h]
  \caption{Set of $2$-dimensional fixed points of $V_{k,n}$}
  \begin{tabular}{cV{3} c | c | c  | c | c | c | c | r |}
    & \multicolumn{3}{c|}{fixed points if $ 0 \leq \theta \leq \theta_{\rm c}$}&\multicolumn{4}{c|}{additional fixed points if $\theta_{\rm c}< \theta <1$}\\ 
    \hlineB{3}
    $k$ even & $(0,0)$ & $(1,0)$ & &$(x_o,y_o)$ && $(x_o,-y_o)$ &\\ \hline
    $k$ odd & $(0,0)$ & $(1,0)$ & $(-1,0)$& $(x_o,y_o)$ & $(-x_o,-y_o)$ & $(x_o,-y_o)$& $(-x_o,y_o)$\\ \hline
  \end{tabular}
  \label{Table1}
\end{table}

\noindent
Only the fixed points $(1,0)$, $(x_o,y_o)$ and $(x_o,-y_o)$ give rise to positive solutions for the Hammerstein equation~\eqref{H_k}.
\end{pro}

Let us finally give the references to the special cases considered prior to this work.  \cite[Theorem 4.2 and Theorem 5.2]{erb} proves the cases $k=2,3$ with $n=1$ of~\eqref{Model} whereas 
in \cite[Theorem 3.2.]{bb} the cases $k\geq 2$ with $n=1$ are given. 
Finally, in \cite[Theorem 2.3]{b1} the cases $k=2$ with general $n \geq 1$ is provided. 
%
%

\section{Proofs}
Note that for the model~\eqref{Model} with $\beta=1$, the kernel $K(t,u)$ of the Hammerstein operator $H_k$ is given by
$$K(t,u)=1+\theta \sqrt[2n+1]{4(t-\frac{1}{2})(u-\frac{1}{2})}.$$

\begin{proof}[Proof of Proposition~\ref{Prop1}] 
Let us start with necessity. Assume $\varphi \in C[0,1]$ to be a solution of the equation~\eqref{H_k}. Then we have
\begin{equation}\label{fi2} \varphi(t)=C_1 + C_2
\theta \sqrt[2n+1]{4(t-\frac{1}{2})},
\end{equation}
where
\begin{equation}\label{c21} 
C_1=\int \limits_0^1 \varphi^k(u)du \quad\text{ and }\quad C_2=\int \limits_0^1 \sqrt[2n+1]{u-\frac{1}{2}}\cdot \varphi^k(u)du.
\end{equation}
Substituting $\varphi(t)$ into the first equation of~\eqref{c21} we get
\begin{equation*}
\begin{split}
C_1&=
\int \limits_0^1 \left(C_1+ C_2 \theta\sqrt[2n+1]{4 (u-\frac{1}{2})}\right)^k du\\
&=\int \limits_0^1 \sum \limits_{i=0}^k \binom{k}{i} C_1^{k-i} \left( C_2 \theta\sqrt[2n+1]{4} \sqrt[2n+1]{u-\frac{1}{2}}\right)^i  du\\
&= \sum \limits_{i=0}^k \binom{k}{i}C_1^{k-i}(\theta C_2 )^i2^{\frac{2i}{2n+1}}\int \limits_0^1 \left(u-\frac{1}{2}\right)^{\frac{i}{2n+1}}  du.
\end{split}
\end{equation*}
Now, we use the following equality
\begin{equation}\label{1}\int \limits_0^1 \left(u-\frac{1}{2}\right)^{\frac{i}{2n+1}}  du =\left\{
                                                  \begin{array}{ll}
                                                    0, & \hbox{if $i$ is odd and} \\
                                                    \frac{2n+1}{2n+1+i} \cdot 2^{- \frac{i}{2n+1}}, & \hbox{if $i$ is even.}
                                                  \end{array}
                                                \right.
                                     \end{equation}
Then we get
                                   $$ C_1=\sum \limits_{j=0}^{\lfloor\frac{k}{2}\rfloor} \binom{k}{2j}
\frac{2n+1}{2n+1+2j}2^{\frac{2j}{2n+1}}C_1^{k-2j} (\theta C_2)^{2j}$$
and substituting the function $\varphi$ into the second equation of~\eqref{c21} we have
\begin{equation*}
\begin{split}
C_2&=\int \limits_0^1 \left(u-\frac{1}{2}\right)^{\frac{1}{2n+1}}\left(C_1+\theta C_2 \sqrt[2n+1]{4 (u-\frac{1}{2})}\right)^k du\cr
&=\int \limits_0^1 \left(u-\frac{1}{2}\right)^{\frac{1}{2n+1}} \sum \limits_{i=0}^k \binom{k}{i}C_1^{k-i} \left(\theta C_2 \sqrt[2n+1]{4} \sqrt[2n+1]{u-\frac{1}{2}}\right)^i  du\cr
&= \sum \limits_{i=0}^k \binom{k}{i}C_1^{k-i}(C_2 \theta)^i2^{\frac{2i}{2n+1}}\int \limits_0^1 \left(u-\frac{1}{2}\right)^{\frac{i+1}{2n+1}}  du.
\end{split}
\end{equation*}
Now, using the following equality
\begin{equation}\label{2}\int \limits_0^1 \left(u-\frac{1}{2}\right)^{\frac{i+1}{2n+1}}  du =\left\{
                                                  \begin{array}{ll}
                                                    0, & \hbox{if $i$ is even and } \\
                                                   \frac{2n+1}{2n+2+i} \cdot 2^{- \frac{i+1}{2n+1}}, & \hbox{if $i$ is odd}
                                                  \end{array}
                                                \right.
                                    \end{equation}
we arrive at the equation
                                   $$ C_2= \sum \limits_{j=0}^{\lfloor\frac{k}{2}\rfloor} \binom{k}{2j+1}
\frac{2n+1}{2n+2+2j+1}2^{\frac{2j}{2n+1}}C_1^{k-2j-1} (\theta C_2)^{2j+1}.$$
In particular, the point $(C_1,C_2) \in \R^2$ must be a fixed point of the operator $V_{k,n}$ from~\eqref{V_k}.

\medskip
For the sufficiency, assume that, a point $(C_1,C_2) \in \R^2$ is a fixed point of the operator $V_{k,n}$ and define the function $\varphi \in C[0,1]$ by the equality $$\varphi(t)=C_1 + C_2
\theta \sqrt[2n+1]{4(t-\frac{1}{2})}.$$
Then, we can calculate
\begin{equation}\label{k}
\begin{split}
&(H_k \varphi)(t)=\int \limits_0^1 \left(1+\sqrt[2n+1]{4} \theta \sqrt[2n+1]{(t-\frac{1}{2})(u-\frac{1}{2})}\right)\varphi^k(u)du\cr
&=\int \limits_0^1 \varphi^k(u)du+\sqrt[2n+1]{4} \theta \sqrt[2n+1]{t-\frac{1}{2}} \int \limits_0^1 \sqrt[2n+1]{u-\frac{1}{2}}\varphi^k(u) du\cr
&=\int \limits_0^1 \left(C_1 + C_2\theta \sqrt[2n+1]{4(u-\frac{1}{2})}\right)^k du\cr
&\hspace{1cm}+ \sqrt[2n+1]{4} \theta \sqrt[2n+1]{t-\frac{1}{2}} \int \limits_0^1 \sqrt[2n+1]{u-\frac{1}{2}} \left(C_1+ C_2 \theta \sqrt[2n+1]{4(u-\frac{1}{2})}\right)^k du\cr
&=\sum \limits_{i=0}^k \binom{k}{i}C_1^{k-i}(\theta C_2 )^i2^{\frac{2i}{2n+1}}\int \limits_0^1 \left(u-\frac{1}{2}\right)^{\frac{i}{2n+1}}  du+ \sqrt[2n+1]{4} \theta \sqrt[2n+1]{t-\frac{1}{2}} \cr 
 &\hspace{1cm}\times \sum \limits_{i=0}^k \binom{k}{i}C_1^{k-i}(C_2 \theta)^i 2^{\frac{2i}{2n+1}}\int \limits_0^1 \left(u-\frac{1}{2}\right)^{\frac{i+1}{2n+1}}  du.
\end{split}
\end{equation}
Now, we using~\eqref{1} and~\eqref{2}, from~ \eqref{k} we get
\begin{equation*}
\begin{split}
(H_k \varphi)(t)&=\sum \limits_{j=0}^{\lfloor\frac{k}{2}\rfloor} \binom{k}{2j}
\frac{2n+1}{2n+1+2j}2^{\frac{2j}{2n+1}}x^{k-2j} (\theta y)^{2j} \cr
&\hspace{0.5cm}+ \theta \sqrt[2n+1]{4(t-\frac{1}{2})} \sum \limits_{j=0}^{\lfloor\frac{k}{2}\rfloor} \binom{k}{2j+1}
\frac{2n+1}{2n+2+2j+1}2^{\frac{2j}{2n+1}} x^{k-2j-1} (\theta y)^{2j+1}\cr
&=C_1+C_2 \theta \sqrt[2n+1]{4(t-\frac{1}{2})}=\varphi(t).
\end{split}
\end{equation*}
Thus, $\varphi$ is a solution of the equation (\ref{H_k}).
\end{proof}

\begin{proof}[Proof of Proposition~\ref{Prop2}]
Let us start by assuming $k$ to be even. We determine the number and form of solutions to $V_{k,n}$ in equation~\eqref{V_k}.
If $\theta\ge 0$, then for $k$ even $(0,0)$ and $(1,0)$ are fixed points.
If $\theta>0$ then there are potentially more fixed points. Indeed, let $\theta>0$ and assume $y>0$ then, writing $z=\theta y/x$, the fixed point equation for~\eqref{V_k} becomes
\begin{equation*}\label{V_z}
\begin{split}
z=\theta\frac{\sum_{i=1,3, \dots, k-1} \binom{k}{i}\frac{2n+1}{2n+2+i} \cdot 2^{\frac{i-1}{2n+1}} \cdot z^{i}}{\sum_{i=0,2, \dots, k}\binom{k}{i}\frac{2n+1}{2n+1+i}2^{\frac{i}{2n+1}}z^{i}}=\theta\frac{F_1(z)}{F_2(z)}=f(z).
\end{split}
\end{equation*}
Hence, in order to find solutions, we have to find roots of the polynomial
\begin{equation}\label{Poly}
\begin{split}
P(z)&=\sum_{i=1,3, \dots, k+1}\binom{k}{i-1}\frac{2n+1}{2n+i}2^{\frac{i-1}{2n+1}}z^{i}-\theta\sum_{i=1,3, \dots, k-1}\binom{k}{i}\frac{2n+1}{2n+2+i}2^{\frac{i-1}{2n+1}}z^{i}\\
&=r_\theta(k,k+1)z^{k+1}+\sum_{i=1,3, \dots, k-1}r_\theta(k,i)z^{i}
\end{split}
\end{equation}
where $r_\theta(k,k+1)=\frac{2n+1}{2n+k+1}2^{\frac{k}{2n+1}}$ and
\begin{equation*}\label{V_t}
\begin{split}
r_\theta(k,i)&=\binom{k}{i}\frac{2n+1}{2n+2+i}2^{\frac{i-1}{2n+1}}\Big[\frac{i}{k-i+1}\frac{2n+2+i}{2n+i}-\theta\Big].
\end{split}
\end{equation*}
Moreover,
\begin{equation*}
\begin{split}
r_\theta(k,i)\left\{
\begin{array}{lllllll}
<0\qquad \text{if}\quad \theta>\frac{i}{k-i+1}\frac{2n+2+i}{2n+i}\\
=0\qquad \text{if}\quad \theta=\frac{i}{k-i+1}\frac{2n+2+i}{2n+i}\\
>0\qquad \text{if}\quad \theta<\frac{i}{k-i+1}\frac{2n+2+i}{2n+i}
\end{array}\right.
\end{split}
\end{equation*}
and we denote the critical $\theta$ by $\theta_{k,i}$. Further note that $i\mapsto\theta_{k,i}$ is increasing. Indeed, the derivative of the continuous version is given by
\begin{equation*}
\begin{split}
\frac{4(1 + k) n^2 + (3 + k) i^2 +  4 (1 + k) n (1 + i))}{(1 + k - i)^2 (2 n + i)^2}
\end{split}
\end{equation*}
which is non-negative.
Hence, for $\theta$ below the lowest critical value, $\theta_{k,1}=\frac{2n+3}{k(2n+1)}$, all coefficients are positive and hence there is no positive real root by Descartes' rule of sign.
Further, again by Descartes' rule of sign, if we increase $\theta>\theta_{k,1}$, then there is exactly one sign change and hence, exactly one non-trivial positive real root which we denote $z_0$.

Since only odd term appear in the polynomial, with $z_0$ also $-z_0$ is a root.
In order to recover a solution $(x,y)$ from the positive non-trivial solution $z_0$, note that
\begin{equation*}\label{V_kn}
\begin{split}
V_{k,n}(x,y)=\left\{
\begin{array}{lllllll}
x=x^k\sum \limits_{i=0,2, \dots, k} \binom{k}{i}
\frac{2n+1}{2n+1+i} \cdot 2^{\frac{i}{2n+1}} \cdot (\frac{\theta y}{x})^i=x^kF_2(\frac{\theta y}{x})$$
\\ [6 mm]$$y=x^k\sum \limits_{i=1,3, \dots, k-1}\binom{k}{i}
\frac{2n+1}{2n+2+i} \cdot 2^{\frac{i-1}{2n+1}} (\frac{\theta y}{x})^i=x^kF_1(\frac{\theta y}{x})
\end{array}
\right.
\end{split}
\end{equation*}
and hence $x_0=F_2(z_0)^{1/(1-k)}>0$ and $y_0=F_1(z_0)F_2(z_0)^{k/(1-k)}>0$ solve the $2$-dimensional equation. Note that $(x_0,y_0)$ is the only solution with $\theta y_0/x_0=z_0$. Indeed, any other such solution would be $x_1=c x_0$ and $y_1=cy_0$ for some $c\in\R\setminus\{0\}$, but plugging this into the first line of the above equation gives $c=c^k$ which is true iff $c=1$ for even $k$.

Further, note that $F_2(-z_0)^{1/(1-k)}=F_2(z_0)^{1/(1-k)}=x_0$ and $F_1(-z_0)F_2(-z_0)^{k/(1-k)}=-F_1(z_0)F_2(z_0)^{k/(1-k)}=-y_0$ and hence also $(x_0,-y_0)$ is a solution to the $2$-dimensional fixed point equation. Using similar arguments one can show that this is the only fixed point with $-\theta y_0/x_0=-z_0$.

\medskip
For odd $k$ and $\theta\ge 0$ we have fixed points $(0,0)$, $(1,0)$ and $(-1,0)$.
For the additional fixed point, the calculations are analogous, but without the leading term $z^{k+1}$, yielding again to fixed point $z_0$ and $-z_0$ for $\theta>\theta_{k,1}$. In contrast to the case for even $k$, for odd $k$, both $(x_0,y_0)$ and $(-x_0,-y_0)$ are $2$-dimensional fixed points corresponding to $z_0$. Finally, the fixed points $(x_0,-y_0)$ and $(-x_0,y_0)$ corresponds to $-z_0$. The complete list of fixed points is recorded in Table~\ref{Table1}.

\medskip
For $(\pm x_0,\pm y_0)$ to give rise to a positive solution, by the form of solutions $\varphi$ we must have that for all $t\in[0,1]$
\begin{equation*}
\pm x_o\pm y_o\theta\sqrt[2n+1]{4(t-1/2)}>0.
\end{equation*}
Clearly, for $-x_o$, in $t=1/2$, the inequality is violated and it suffices to consider the points  $(x_0,\pm y_0)$. By monotonicity in $t$, it suffices to show that
\begin{equation}\label{Criterium}
2^{-1/(2n+1)}>\theta y_o/x_o=z_o
\end{equation}
for the positive solution of the polynomial $P$ from~\eqref{Poly}. Since, the sign change in the polynomial must be from minus to plus, we need to determine its sign in $2^{-1/m}$ where we put $m=2n+1$. We show that indeed $P(2^{-1/m})>0$ which implies that~\eqref{Criterium} is satisfied and thus $(x_o,\pm y_o)$ correspond to positive solutions. 
Note that
\begin{equation*}
\begin{split}
P(2^{-1/m})&=2^{-1/m}\sum_{i=1,3, \dots, k+1}\binom{k}{i-1}\frac{m}{m+i-1}-\theta\sum_{i=1,3, \dots, k-1}\binom{k}{i}\frac{m}{m+1+i}>0
\end{split}
\end{equation*}
is implied by 
\begin{equation*}
\begin{split}
\sum_{i=1,3, \dots, k+1}\binom{k}{i-1}\frac{1}{m+i-1}-\sum_{i=1,3, \dots, k-1}\binom{k}{i}\frac{1}{m+1+i}>0.
\end{split}
\end{equation*}
We can further bound the left hand side from below by
\begin{equation*}
\begin{split}
\sum_{i=0}^k(-1)^i\binom{k}{i}\frac{1}{m+1+i}=\frac{k!m!}{(m+1+k)!}
\end{split}
\end{equation*}
which is positive for all $m, k$. This completes the proof. 
\end{proof}
%
%

\section{Acknowledgement}
Golibjon Botirov thanks the DAAD program for the financial support and the Weierstrass Institute Berlin for its hospitality. Benedikt Jahnel thanks the Leibniz program 'Probabilistic methods for mobile ad-hoc networks' for the support.

\end{document}